\documentclass[preprint,numbers,12pt]{elsarticle}
\usepackage{amsthm,amsmath,amsfonts}
\usepackage{amssymb}
\usepackage{color,graphicx}
\usepackage{hyperref}
\usepackage{enumerate}
\usepackage{multirow}
\usepackage{cases}

\newcommand{\R}{\mathbb{R}}

\providecommand{\abs}[1]{\lvert#1\rvert}

\providecommand{\diag}{{\rm diag}}

\newcommand{\lift}[1]{{\uparrow \!\!  #1}}
\newcommand{\rs}[1]{{\mbox{\scriptsize \sc #1}}}

\newcommand{\sr}[1]{{\cal #1}}
\newcommand{\dd}[1]{\mathbb{#1}}
\newcommand{\ray}{{\rm r}}

\newcommand{\br}[1]{\langle #1 \rangle}
\newcommand{\brb}[1]{\big\langle #1 \big\rangle}

\newtheorem{lemma}{Lemma}

\newtheorem{corollary}{Corollary}
\newtheorem{theorem}{Theorem}
\newtheorem{question}{Question}
\theoremstyle{definition}
\numberwithin{equation}{section}

\newtheorem{remark}{Remark}

\newcommand{\eq}[1]{(\ref{eq:#1})}
\newcommand{\eqn}[1]{(\ref{eqn:#1})}
\newcommand{\lem}[1]{Lemma~\ref{lem:#1}}
\newcommand{\cor}[1]{Corollary~\ref{cor:#1}}
\newcommand{\thr}[1]{Theorem~\ref{thr:#1}}

\newcommand{\que}[1]{Question~\ref{que:#1}}

\newcommand{\rem}[1]{Remark~\ref{rem:#1}}
\newcommand{\fig}[1]{Figure~\ref{fig:#1}}
\newcommand{\app}[1]{\ref{app:#1}}
\newcommand{\sectn}[1]{Section~\ref{sect:#1}}

\newcommand{\thrt}[1]{\ref{thr:#1}}

\newcommand{\sect}[1]{\ref{sect:#1}}

\journal{Stochastic Processes and Their Applications}

\begin{document}

\begin{frontmatter}

\title{Decomposable stationary distribution of\\ a multidimensional SRBM}

\author[cu]{J. G. Dai}
\ead{jim.dai@cornell.edu}

\author[tus]{Masakiyo Miyazawa\corref{cor1}}
\ead{miyazawa@rs.tus.ac.jp}

\author[cu]{Jian Wu}
\ead{jw926@cornell.edu}

\cortext[cor1]{Corresponding author}

\address[cu]{School of Operations Research and Information Engineering, Cornell University, Ithaca, NY 14853}
\address[tus]{Department of Information Sciences, Tokyo University of Science, Noda, Chiba 278-8510, Japan}

\begin{abstract}
We focus on the stationary distribution of a multidimensional semimartingale reflecting Brownian motion (SRBM) on a nonnegative orthant. Assuming that the stationary distribution exists and is decomposable---equal to the product of two marginal distributions, we prove that these marginal distributions are the stationary distributions of some lower dimensional SRBMs, whose data can be explicitly computed through that of the original SRBM. Thus, under the decomposability condition, the stationary distribution of a high dimensional SRBM can be computed through those of lower dimensional SRBMs.  Next, we derive necessary and sufficient conditions for some classes of SRBMs to satisfy the decomposability condition.
\end{abstract}


\end{frontmatter}

\section{Introduction}
\label{sect:introduction}
We are concerned with a $d$-dimensional semimartingale reflecting
Brownian motion (SRBM) that lives on the nonnegative orthant
$\mathbb{R}^d_+$, where $\dd{R}_{+}$ is the set of all nonnegative
real numbers. The SRBM is specified by a $d \times d$ covariance matrix
$\Sigma$, a drift vector $\mu\in \dd{R}^d$, and a $d \times d$ reflection
matrix $R$. Namely, $(\Sigma, \mu, R)$ is the modeling primitives of
the SRBM on $\mathbb{R}^d_+$. As usual, we assume that $\Sigma$ is
positive definite and $R$ is completely-$\sr{S}$. See \app{matrix}
for the definitions of matrix classes used in this paper. A $(\Sigma, \mu,
R)$-SRBM $Z$ has the following semimartingale representation:
\begin{eqnarray}
&&  Z(t) = Z(0) +  X(t) + RY(t) \in \dd{R}^d_+, \quad t\ge 0, \label{eq:RBM1}\\
&&  X=\{X(t), t\ge 0\} \text{ is a  $(\Sigma, \mu)$-Brownian motion},\qquad\label{eq:RBM2} \\
&&  Y(0)=0, Y(\cdot) \text{ is nondecreasing,}\label{eq:RBM3}\\
&&  \int_0^\infty Z_i(t) d Y_i(t) =0 \text{ for } i=1, \ldots, d. \label{eq:RBM4}
\end{eqnarray}
For the complete definition of an SRBM $Z=\{Z(t); t \ge 0\}$, we refer to Section~A.1 of \cite{DaiMiya2011} (see \cite{TaylWill1993,Will1995} for more details). 

Our focus is on the stationary distribution of the $d$-dimensional
SRBM. 
Throughout this paper, we
 assume that the stationary distribution exists. As a consequence, the
 primitive data satisfies the following condition:
\begin{eqnarray}
\label{eq:stability}
  \mbox{$R$ is nonsingular, and $R^{-1} \mu < 0$,}
\end{eqnarray}
where vector inequality is interpreted entrywise. If $R$ is either a $\sr{P}$-matrix for $d=2$ or an $\sr{M}$-matrix for
an arbitrary $d\ge 1$, then this condition is known to be sufficient, but generally not for $d \ge 3$ (see, e.g.,\cite{BramDaiHarr2010}).

Let $J \equiv \{1,2,\ldots, d\}$. A pair $(K, L)$ is said to be a
partition of $J$ if $K\cup L=J$ and $K\cap L=\emptyset$.  We consider
conditions for the stationary distribution of an SRBM in $\dd{R}^d_+$ to
be the product of two 
marginal distributions associated with a partition $(K,L)$
of the set $J$. Such a stationary distribution is said to be
\emph{decomposable} with respect to $K$ and $L$. We have two major
contributions for the decomposability of the stationary distribution.

We first characterize, in \thr{marginals}, two marginal distributions
associated with a partition $(K,L)$  under the decomposability
assumption.  We prove that they are the stationary distributions of
some $|K|$- and $|L|$-dimensional SRBMs, where $|U|$ denotes the cardinality
of a set $U$. We also identify the data for these lower dimensional
SRBMs. Thus, under the decomposability assumption, we can obtain the
stationary distribution of the original SRBM by computing those of the
lower dimensional ones.

However, this characterization of the marginal distributions is not
sufficient for the decomposability. So, we next consider necessary and
sufficient conditions for the decomposability. We obtain those
conditions for several classes of SRBMs (\thr{decomposition} and
\cor{decomposition 1}). These classes include diffusion limits of
tandem queues and of queueing networks that have two sets of nodes
with feed-forward routing between these two sets. Note that the
decomposability does not mean a complete separation of such a network
into two subnetworks. We illustrate 
a tandem queue next. 

Consider a $d$-station generalized Jackson network in series, which is
referred to as a tandem queue. In this tandem queue, the interarrival
times to station $1$ are assumed to be iid with mean $1/\beta_{0}$ and
coefficient of variation (CV) $c_0$. The service times at
station $i$ are assumed to be iid with mean $1/\beta_{i}$ and CV
$c_i$, $i\in J$ (see \fig{dd-tandem}).  
\begin{figure}[h]
 	\centering
	\includegraphics[height=2.5cm]{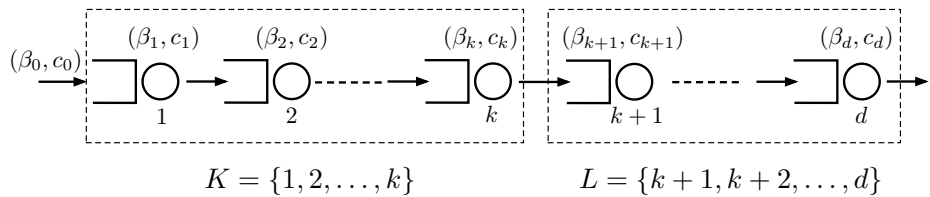}
	\caption{The $d$-station tandem queue partitioned into two blocks}
	\label{fig:dd-tandem}
\end{figure} 
The diffusion limit of queue length process for this tandem queue is known to be the $d$-dimensional SRBM with the following reflection matrix $R$, the covariance matrix $\Sigma$, and the drift vector $\mu$ (see, e.g., Section 7.5 of \cite{ChenYao2001}).
\begin{eqnarray}
\label{eq:tandem R}
 && R_{ij} = \left\{\begin{array}{ll}
  1, \quad & i=j \mbox{ for } i=1,2,\ldots, d,\\
  -1, \quad & j = i-1 \mbox{ for } i=2,\ldots,d,\\
  0, \quad & \mbox{otherwise},
  \end{array} \right.
\end{eqnarray}
\begin{eqnarray}
\label{eq:tandem Sigma}
 && \Sigma_{ij} =  \beta_0\times
 \begin{cases}
 c_{i-1}^{2} + c_{i}^{2},  & i=j \mbox{ for } i=1,2,\ldots, d,\\
  -c^{2}_{i-1},  & j = i-1 \mbox{ for } i=2,\ldots,d,\\
  -c^{2}_{i},  & j = i+1 \mbox{ for } i=1,\ldots,d-1,\quad\\
  0, & \mbox{otherwise},
 \end{cases}\\
\label{eq:tandem mu}
 && \mu_{i} = \beta_{i-1} - \beta_{i}  \text{ for }  i =1, \ldots, d.
\end{eqnarray}
For example, when $d=3$, $R$ and $\Sigma$ are 
 given by 
\begin{equation*}
  R= \begin{pmatrix}
    1 & 0 & 0 \\
   -1 & 1 & 0 \\
   0  & -1 & 1 
  \end{pmatrix}, 
\quad 
  \Sigma =  \beta_0
\begin{pmatrix}
  c^{2}_0+c^{2}_1 & -c^{2}_1 & 0 \\ 
 -c^{2}_1 & c^{2}_1+c^{2}_2  & -c^{2}_2\\
  0 & -c^{2}_2 & c^{2}_2+ c^{2}_3 
\end{pmatrix}.
\end{equation*}
We assume  that
$\Sigma$ is nonsingular and  condition \eq{stability} is satisfied,
which is equivalent to $\beta_{0} < \beta_{i}$ for $i=1,2,\ldots,d$. 
It follows from \cite{HarrWill1987} that the SRBM $Z$ has a
unique stationary distribution $\pi$. For each $k \in \{1, 2, \ldots, d-1\}$, let
\begin{eqnarray*}
  K \equiv \{1,2,\ldots, k\}, \qquad L \equiv \{k+1,\ldots, d\}.
\end{eqnarray*}
See \fig{dd-tandem} for an illustration of these sets. It will be shown
in \cor{decomposition 1} that if 
\begin{eqnarray}
\label{eq:de 1}
 && c_0=c_i \quad \text{ for } i = 1, \ldots, k,
\end{eqnarray}
then the $d$-dimensional stationary distribution is decomposable with respect to $K$
and $L$.

We are motivated by applications of SRBMs and our recent work
\cite{DaiMiyaWu2014a}. Multidimensional SRBMs have been widely used
for queueing applications \cite{Will1996} and some areas including 
mathematical finance \cite{ShkoKara2013}. For these applications, it
is important 
to obtain the stationary distribution in a 
tractable form. However, this is a very hard problem even for
$d=2$ unless 
the $d$-dimensional SRBM has a product form stationary distribution.
A multidimensional distribution is said to have a product form if it
is the product of one dimensional marginal distributions. 

It is shown in \cite{HarrWill1987a}  that this product form holds 
for the stationary distribution of an SRBM $Z$ if
and only if the following skew symmetry condition  
\begin{eqnarray}
  \label{eq:skew symmetric}
   2\Sigma =R\diag(R)^{-1}\diag(\Sigma) + \diag(\Sigma)\diag(R)^{-1}R^{\rs{t}}
\end{eqnarray}
is satisfied, where for a matrix $A$, $\diag(A)$ denotes the diagonal
matrix whose entries are diagonals of $A$, and $A^{\rs{t}}$ denotes
the transpose of $A$. Furthermore, under \eq{skew symmetric}, the
one-dimensional marginal stationary distribution in the $i$th
coordinate has the exponential 
distribution with mean $1/\alpha_{i}$, where column vector $\alpha
\equiv (\alpha_{1}, \ldots, \alpha_{d})^{\rs{t}}$ is given by 
\begin{eqnarray}
\label{eq:alpha}
  \alpha = -2\diag(\Sigma)^{-1}\diag(R)R^{-1}\mu.
\end{eqnarray}

Thus, the stationary distribution of $Z$ is explicitly obtained under
the skew symmetry condition \eq{skew symmetric}.  However, the
condition \eq{skew symmetric} may be too strong in some
applications. In particular, \eq{alpha} is independent of covariances
$\Sigma_{ij}$ for $i \ne j \in J \equiv \{1,2,\ldots,d\}$. Still, 
product form based approximation is often used even though
its accuracy cannot be assessed when  condition
\eq{skew symmetric} is not satisfied; see, for example,
\cite{KangKellLeeWill2009}.

This product form based approximation may be improved by the decomposability. For example, let us consider the SRBM for the tandem queue depicted in \fig{dd-tandem}, and assume the decomposability condition \eq{de 1}. Then, its $\abs{K}$-dimensional marginal is of product-form, which can be computed easily. 
Its $\abs{L}$-dimensional marginal is the stationary distribution of 
an $\abs{L}$-dimensional SRBM. If one adopts the algorithm in
\cite{DaiHarr1992} to compute the  
stationary distribution of a $d$-dimensional SRBM,  the computational
complexity  
is in the order of $d^{3n}$, where $n$ is the degree of
polynomials used in the numerical approximation; see the discussion on
page 78 of \cite{DaiHarr1992}.
 Therefore, using our decomposition result, the original $d$-dimensional
stationary distribution can be computed in the order of
$\abs{L}^{3n}+d$ operations, a huge saving as compared with $d^{3n}$
when $\abs{L}$ is small or moderate.

In recent years, for $d=2$, the tail asymptotics of the stationary
distribution including decay rates have been well studied (e.g., see
\cite{AvraDaiHase2001,DaiMiya2011,DaiMiya2013}) even though condition
\eq{skew symmetric} is not satisfied.  Those decay rates may be used
for better approximations of a two-dimensional stationary distribution
as recently shown for a two dimensional reflecting random walk in
\cite{LatoMiya2014}. We hope that such a two-dimensional approximation
can be used to develop better approximations for the stationary
distribution of a high dimensional SRBM. Thus, weaker conditions than the product form will be useful to facilitate an approximation for a higher dimensional SRBM. A decomposability condition that is checkable by modeling primitives will considerably widen the applicability of a multidimensional SRBM. 

We are also inspired by geometric interpretations in
\cite{DaiMiyaWu2014a} for the product form characterization. That work
focuses on the product form, but considers characterizations that are
different from the skew symmetric condition \eq{skew symmetric}. Among
them, Corollary 2 of \cite{DaiMiyaWu2014a} shows that, for each pair of
$i \ne j \in J$, the corresponding two-dimensional, marginal
distribution is equal to the stationary distribution of some
two-dimensional SRBM if the $d$-dimensional stationary distribution
has a product form, that is, condition \eq{skew symmetric} is
satisfied. This motivates us to consider the lower dimensional SRBMs
corresponding to the marginal distributions under the decomposability.

Other than the product form case, there are some two-dimensional
  SRBMs whose stationary distributions are obtained in closed
  form. For example, in \cite{DiekMori2009,Fosc1982}, the
    stationary densities have finite
  sums of exponential functions. However, those results are only
  obtained for very limited cases, and their expressions for the stationary densities are quite
  different from decomposability. Thus, we do not have immediate
  benefit of using them to study decomposability. However, we may
  consider them as another type of approximation. Thus, it may be
  interesting to combine them with decomposability to get better
  approximations.

This paper consists of four sections. We present our results,
Theorems \thrt{marginals} and \thrt{decomposition} and their
corollaries, in \sectn{main}. We discuss basic facts and preliminary
results in \sectn{characterizing}. Theorems \thrt{marginals} and
\thrt{decomposition} are proved in Sections \sect{theorem 1} and
\sect{theorem 2}, respectively. We finally remark some future work in
\sectn{concluding}. 

We will use the following notation unless otherwise stated.
\begin{table}[htdp]
\label{tab:primitive}
\begin{center}
\begin{tabular}{ll}
  $J$ & $\{1,2,\ldots, d\}$, and, for $U, V \subset J$\\
  $A^{(U,V)}$ & the $\abs{U} \times \abs{V}$ submatrix of a $d$-dimensional square matrix $A$ whose\\
  & row and column indices are taken from $U$ and $V$, respectively.\\
  $x^{U}$ & $|U|$-dimensional vector with $x^{U}_{i} = x_i$ for $i\in U$,\\
  & where $x^{U}_{i}$ is the $i$-th entry of $x^{U}$\\
  $\lift{x^{U}}$ & $d$-dimensional vector $x$ with $x_{i} = x^{U}_{i}$ for $i \in U$\\
  & and $x_{i} = 0$ for $i \in J \setminus U$\\
  $f^{U}(x^{U})$ & $f(\lift{x^{U}})$ for function $f$ from $\dd{R}^{d}$ to $\dd{R}$\\
   $\br{x, y}$ & $\sum_{i=1}^{d} x_{i} y_{j}$ for $x, y \in \dd{R}^{d}$
\end{tabular}
\caption{A summary of basic notation}
\end{center}
\label{notation 1}
\end{table}

\pagebreak

For example, for $i \in U$, the $i$-th entry of $A^{(U,V)} x^{V}$
is $\sum_{j \in V} [A^{(U,V)}]_{ij} x^{V}_{j}$, where both
$T_{ij}$ and $[T]_{ij}$ denote the $(i,j)$-entry of a matrix $T$.

\section{Main results}
\label{sect:main}

To state our main results, we need the following notation. Let
\begin{eqnarray*}
  Q = R^{-1}.
\end{eqnarray*}
 For a  non-empty set $U \subset J$ such that $Q^{(U,U)}$ is invertible, define
\begin{eqnarray}
 && \tilde{\Sigma}(U) = (Q^{(U,U)})^{-1} (Q \Sigma Q^{\rs{t}})^{(U,U)}
 ((Q^{(U,U)})^{-1})^{\rs{t}}, \quad \label{eq:SigmaU}\\
 && \tilde{\mu}(U) = (Q^{(U,U)})^{-1} (Q \mu)^{U},\label{eq:muU}\\
 && \tilde{R}(U) = (Q^{(U,U)})^{-1}. \label{eq:RU}
\end{eqnarray}

\begin{theorem}
\label{thr:marginals}
  Assume that $R$ is completely-${\cal S}$ and that the
  $d$-dimensional SRBM $Z=\{Z(t), t\ge 0\}$ has a stationary
  distribution $\pi$. Let $K$ and $L$ be a non-empty partition of $J$. Assume
  that $Z^{K}(0)$ and $Z^{L}(0)$ are independent under the stationary
  distribution $\pi$. Then, for $U = K$ and $U=L$,\\ 
 (a) $Q^{(U,U)}$ is invertible, and its inverse matrix $\tilde{R}(U)$ is completely-${\cal S}$.\\
 (b) The $|{U}|$-dimensional $(\tilde{\Sigma}(U)$, $\tilde{\mu}(U)$, $\tilde{R}(U))$-SRBM has a stationary distribution that is equal to the distribution of $Z^{U}(0)$ under $\pi$.
\end{theorem}

\begin{remark}
\label{rem:marginals 0}
  If $R$ is a $\sr{P}$-matrix, then $R$ is invertible and
  $R^{-1}$ is also a $\sr{P}$-matrix 
  by Lemma 6 of \cite{DaiMiyaWu2014a}.
Therefore part (a) is immediate when $R$ is a ${\cal P}$-matrix
  since each $\sr{P}$-matrix is  completely-$\sr{S}$. However,
even under condition \eq{stability},
 the completely-${\cal S}$ property of $R$ does not imply that
 $R^{-1}$ has the same property as shown by the following
  example:
\begin{eqnarray*}
  R = \left(\begin{array}{cc} 1 & 2 \\1& 1\end{array}\right) \mbox{ implies } 
  R^{-1} = \left(\begin{array}{cc} -1 & 2 \\ 1& -1\end{array}\right),
\end{eqnarray*}
where $R$ is  completely-$\sr{S}$, but $R^{-1}$ is not.
 Thus, part (a) of the theorem is not obvious at all. \\
\end{remark}

\thr{marginals} will be proved in \sectn{theorem 1}, and the corollary below is proved in \app{corollary 1}.

\begin{corollary}
\label{cor:marginal 1}
  Under the same assumptions of \thr{marginals}, for $i \in J$, assume
  that $Z^{\{i\}}(0)$ and $Z^{J \setminus \{i\}}(0)$ are independent
  under the stationary distribution $\pi$. Then, $Z^{\{i\}}(0)$ has
  the exponential distribution with mean $1/\lambda_i$,
where 
\begin{eqnarray}
  \label{eq:lambdai}
&&  \lambda_i = \Delta_i Q_{ii}, \\
&&   \Delta_i = -\frac{2 \langle \mu, (Q^{\rs{t}})^{(i)} \rangle}{\langle
    (Q^{\rs{t}})^{(i)},
    \Sigma (Q^{\rs{t}})^{(i)} \rangle}>0,\label{eq:Deltai}
\end{eqnarray}
and $(Q^{\rs{t}})^{(i)}$ is the $i$th column of $Q^{\rs{t}}$. 
\end{corollary}

\begin{remark}
\label{rem:marginal 1}
The parameter $\lambda_i$ in (\ref{eq:lambdai}) has a geometric
interpretation; see (\ref{eq:thataray}) in the proof of \cor{marginal 1}.
Note that $\lambda_i$  uses information on covariance $\Sigma_{ij}$
in general, so it may be different from $\alpha_{i}$ of \eq{alpha},
although we must have $\alpha_{i} = \lambda_i$ if the
stationary distribution has a product form. This fact can
  quickly be checked by using Lemma 2 in \cite{DaiMiyaWu2014a}, which characterizes the skew symmetric condition in terms of moment generating functions (see \app{corollary 1} for details). Since the exponential
distribution in \cor{marginal 1} is obtained under the weaker
condition than the product form condition \eq{skew symmetric} for $d
\ge 3$, it is intuitively clear that one should use $\lambda_{i}$ instead of
$\alpha_{i}$ in the product form approximation of a stationary
distribution.  We will  further discuss this issue in
\sectn{concluding}. 
\end{remark}

In general, (a) and (b) of \thr{marginals} are necessary but not sufficient for
$Z^{K}(0)$ and $Z^{L}(0)$ to be independent under the stationary
distribution $\pi$. For example, if $J = \{1,2\}$, then the marginal
exponential distributions are determined by the mean $1/\lambda_i$ for
$i=1,2$ by \cor{marginal 1}, but these marginals are not sufficient
for the skew symmetric condition, which is equivalent to that
$Z^{\{1\}}(0)$ and $Z^{\{2\}}(0)$ are independent. This is because
a condition weaker than the decomposability condition is used in the
proof of 
\thr{marginals} and 
therefore of \cor{marginal 1}. This fact will be detailed in
\sectn{concluding}.  Thus, it would be interesting to seek additional conditions which imply the decomposability. However, to
identify these extra conditions is
generally a hard problem, so we consider a relatively simple
situation. For this, we consider SRBMs arising from queueing networks
that have two sets of stations with feed-forward routing between these two
sets. 

\begin{theorem}
\label{thr:decomposition}
  Assume that $R$ is completely-${\cal S}$ and that the
  $d$-dimensional SRBM $Z=\{Z(t), t\ge 0\}$ has a stationary
  distribution $\pi$. Let $K$ and $L$ be a non-empty partition of $J$, and assume that
\begin{eqnarray}
\label{eq:forward condition}
  R^{(K,L)} = 0.
\end{eqnarray}
If $Z^{K}(0)$ and $Z^{L}(0)$ are independent under $\pi$ and if $Z^{K}(0)$ is of product form under $\pi$, then 
\begin{eqnarray}
\label{eq:condition1}
 \lefteqn{2\Sigma^{(K,K)} = R^{(K,K)}\diag(R^{(K,K)})^{-1}\diag(\Sigma^{(K,K)})} \nonumber\\
&& \hspace{10ex} + \diag(\Sigma^{(K,K)})\diag(R^{(K,K)})^{-1}(R^{(K,K)})^{\rs{t}}, \\
\label{eq:condition2}
 \lefteqn{2\Sigma^{(L,K)} = R^{(L,K)}\diag(\Sigma^{(K,K)})\diag(R^{(K,K)})^{-1}.}
\end{eqnarray}
Conversely, if $\Sigma$ and $R$ satisfy (\ref{eq:condition1}) and (\ref{eq:condition2}),
  and if the $\abs{L}$-dimensional $(\Sigma^{(L,L)}, \tilde{\mu}(L),
  R^{(L,L)})$-SRBM has a stationary distribution, then $Z^K(0)$ and
  $Z^L(0)$ are independent under $\pi$ and $Z^K(0)$ is of product form
  under $\pi$.
\end{theorem}

\begin{remark}
\label{rem:decomposition 1}
  After an appropriate recording of the coordinates, the condition \eq{forward condition} can be written as
\begin{eqnarray*}
  R= \begin{pmatrix}
    R^{(K,K)} & 0 \\
   R^{(L,K)} & R^{(L,L)}
  \end{pmatrix}.
\end{eqnarray*}
  In this case, the covariance matrix $\Sigma$ and the drift vector $\mu$ are partitioned as
\begin{eqnarray*}
  \Sigma =
\begin{pmatrix}
  \Sigma^{(K,K)} & (\Sigma^{(L,K)})^{\rs{t}} \\ 
 \Sigma^{(L,K)} & \Sigma^{(L,L)}\\
\end{pmatrix}, \qquad \mu= (\mu^K, \mu^{L})^{\rs{t}}.
\end{eqnarray*}
Note that \eq{condition1} is the skew symmetric condition for
  $(\Sigma^{(K,K)}, \mu^{K}, R^{(K,K)})$-SRBM. A geometric
  interpretation for this condition is given in Theorem 1 and its
  corollaries in \cite{DaiMiyaWu2014a}. One wonders if a
  similar interpretation can be found
for \eq{condition2}; so far we have not been able to obtain one.  
\end{remark}

This theorem is proved in \sectn{theorem 2}. The next corollary is for an SRBM arising from the $d$ station tandem queue, which was discussed in \sectn{introduction} (see \fig{dd-tandem}). We omit its proof because it is an immediate consequence of \thr{decomposition}.

\begin{corollary}
\label{cor:decomposition 1}
Assume that the $(\Sigma, \mu, R)$-SRBM has a stationary distribution
$\pi$, where  the reflection matrix
$R$, the covariance 
matrix $\Sigma$, and the drift vector $\mu$ are given by \eq{tandem
  R}, \eq{tandem Sigma} and \eq{tandem mu}, respectively.
 For each positive integer $k \le
d-1$, set $K=\{1,\cdots,k\}$ and $L=J\setminus K$. Then $Z^{K}(0)$ and
$Z^{L}(0)$ are independent under $\pi$ if $c_0=c_1=\cdots=c_k$. Furthermore,
for $k=1$, $c_{0} = c_{1}$ is also necessary for this
decomposability. 
\end{corollary}

\section{Proofs of main results}
\label{sect:proofs}

We will prove Theorems \thrt{marginals} and \thrt{decomposition}. For this, we first discuss about equations to characterize the stationary distribution and some basic facts obtained from the decomposability.

\subsection{The stationary distribution}
\label{sect:characterizing}

Assume the SRBM has a stationary distribution. The stationary
distribution must be unique \cite{DaiHarr1992}.
Our first tool is the basic
adjoint relationship (BAR) that characterizes
the stationary distribution.  For this, we first introduce
the boundary measures for a distribution $\pi$ on $(\dd{R}_{+}^{d},
\sr{B}(\dd{R}_{+}^{d}))$, where $\sr{B}(\dd{R}_{+}^{d})$ is the Borel
$\sigma$-field on $\dd{R}_{+}^{d}$. They are defined as
\begin{eqnarray*}
  \nu_{i}(B)=\dd{E}_{\pi} \left[ \int_0^1  1\{ Z(t) \in B\} dY_{i}(t)\right], \quad B \in \sr{B}(\dd{R}_{+}^{d}), \; i\in J.
\end{eqnarray*}
Our BAR is  in terms of moment
generating functions, which  are defined as  
\begin{eqnarray*}
 && \varphi(\theta) = \dd{E}_{\pi}[ e^{\langle \theta, Z(0)\rangle}],\\
 && \varphi_{i}(\theta) = \dd{E}_{\pi} \left[ \int_0^1  e^{\langle
      \theta, Z(t)\rangle}dY_{i}(t)\right], \quad i\in J,
\end{eqnarray*}
where $\dd{E}_\pi$ is the expectation operator when $Z(0)$ has the distribution $\pi$.

Because for each $i\in J$, $Y_{i}(t)$ increases only when $Z_{i}(t)=0$, one has $\varphi_{i}(\theta)$ depends on
$\theta^{J\setminus\{i\}}$ only. Therefore,  
\begin{displaymath}
\varphi_{i}(\theta)=\varphi_{i}\bigl(\lift{\theta^{J\setminus\{i\}}}\bigr).
\end{displaymath}
Note that $Y_{i}(t)$ and $Y_{j}(t)$ may simultaneously increase for $j \ne i$ but it occurs with probability 0 by Lemma 4.5 of \cite{DaiWill1995}. So, we do not need to care about the other $\theta_{j}$'s in the variable $\theta$ of $\varphi_{i}(\theta)$.

For a $(\Sigma, \mu, R)$-SRBM, its data can be alternatively described
in terms of 
$d$-dimensional polynomials, which are defined as
\begin{eqnarray*}
 && \gamma(\theta)=-\frac{1}{2}\langle\theta, \Sigma\theta\rangle-\br{\mu, \theta}, \qquad \theta \in \dd{R}^{d},\\
 && \gamma_{i}(\theta) = \brb{R^{(i)}, \theta}, \qquad \theta \in \dd{R}^{d}, \quad i \in J,
\end{eqnarray*}
where $R^{(i)}$ is the $i$th column of the reflection matrix $R$. Obviously, those polynomials uniquely determine the primitive data, $\Sigma$, $\mu$ and $R$. Thus, we can use those polynomials to discuss everything about the SRBM instead of the primitive data themselves.

The following lemma is critical in our analysis. Equation \eq{key 1} below is the moment 
generating function version of the standard basic adjoint
relationship. We still refer to it
as BAR.

\begin{lemma}
\label{lem:key 1}
(a) Assume $\pi$ is the stationary distribution of a $(\Sigma, \mu,
R)$-SRBM.  For $\theta\in \dd{R}^d$, $\varphi(\theta)<\infty$ implies
$\varphi_i(\theta)<\infty$ for $i\in J$. Furthermore,
\begin{eqnarray}
    \label{eq:key 1}
    \gamma(\theta) \varphi(\theta)= \sum_{i=1}^d \gamma_{i}(\theta)
    \varphi_{i}\bigl(\theta\bigr)
\end{eqnarray}
holds for $\theta \in \dd{R}^{d}$ such that $\varphi(\theta) < \infty$.
(b) Assume that $\pi$ is a probability measure on $\dd{R}^d_+$ and that
$\nu_i$ is a positive finite measure whose support is contained in
$\{x\in\dd{R}^d_+: x_i=0\}$ for $i\in J$. Let $\varphi$ and $\varphi_i$
be the moment generating functions of $\pi$ and $\nu_i$,
respectively. If $\varphi$, 
$\varphi_1$, $\ldots$, $\varphi_d$ satisfy (\ref{eq:key 1}) for each $\theta\in
\dd{R}^d$ with $\theta\le 0$, then $\pi$ is the stationary distribution
and $\nu_i$ is the corresponding boundary measure on $\{x\in\dd{R}^d_+: x_i=0\}$.
\end{lemma}

For (a), the fact that (\ref{eq:key 1}) holds for $\theta\le 0$ is a
special case of the standard basic adjoint relationship (BAR); see,
e.g., equation (7) in \cite{DaiHarr1992}. When some components
of $\theta$ are allowed to be positive in (\ref{eq:key 1}), readers
are referred to the proof of Lemma 4.1 (a) in \cite{DaiMiya2011} to
see how to rigorously derive the relationship. For (b), we refer to
Theorem 1.2 of \cite{DaiKurt1994} (see also \cite{KangRama2012} for a
more general class of reflecting processes). In \cite{DaiKurt1994},
BAR \eq{key 1} is given using differential operators; see, for
  example, again  equation (7) in \cite{DaiHarr1992}.
Under the condition of our lemma, that BAR (7) is satisfied for all
functions  $f$ of the form
\begin{equation}
  \label{eq:2}
  f(x)= e^{\langle \theta, x\rangle} \qquad
  x\in \dd{R}^d_+, \mbox{ for each } \theta \in \dd{R}^d \mbox{ with } \theta \le 0.
\end{equation}
By the analytic extension, BAR (7) in \cite{DaiHarr1992}  continues to
hold for functions 
$f$ when $\theta$ is replaced by $(z_1, \ldots, z_d)$ 
where each $z_i$ is a complex variable with $\Re z_i\le 0$.
This means that BAR is satisfied for finite Fourier
transforms.  Since a continuous
function with a compact support is uniformly approximated by a
sequence of finite Fourier transforms, one can argue that BAR (7) in
\cite{DaiHarr1992} 
holds for all bounded functions whose first- and second- order
derivatives are continuous and bounded. A complete proof for the
equivalence of \eq{key 1} to BAR (7) in \cite{DaiHarr1992}  can be found in the appendix of the  arXiv version of \cite{DaiMiyaWu2014a}.

In the rest of this paper, whenever we write $\varphi(\theta)$, we implicitly
assume it is finite.

Let $K$ and $L$ be a non-empty partition of $J$. In this paper, we
consider conditions for $Z^{K}(0)$ to be independent of $Z^{L}(0)$
under the stationary distribution $\pi$. The independence is equivalent to
\begin{eqnarray}
\label{eq:decomposition 1}
  \varphi(\theta) = \varphi^{K}(\theta^{K}) \varphi^{L}(\theta^{L}),
\end{eqnarray}
where, for  $U \subset J$,
\begin{eqnarray*}
  \varphi^{U}(\theta^{U}) = \varphi(\lift{\theta^{U}}).
\end{eqnarray*}

The next lemma shows how the boundary measure is decomposed under
\eq{decomposition 1}. 
\begin{lemma}
\label{lem:partial independence}
Let $K$ and $L$ be a non-empty partition of $J$.
Assume that $Z^K(0)$ and $Z^L(0)$ are independent under the stationary
distribution $\pi$. Assume $\varphi(\theta)<\infty$.
 Then
\begin{eqnarray}
\label{eq:Palm 1}
  \varphi_{j}(\theta) = \varphi_{j}^{K}(\theta^{{K}}) \varphi^{L}(\theta^{L}), \qquad j \in K,
\end{eqnarray}
where, for $ U \subset J$,
\begin{eqnarray*}
  \varphi^{U}_{j}(\theta^{U}) = \varphi_{j}(\lift{\theta^{U}}).
\end{eqnarray*}
\end{lemma}
\begin{proof}
We first prove that for $i \ne j$,
\begin{equation}
  \label{eq:thetajlimit}
  \lim_{\theta_{j} \downarrow -\infty} \varphi_{i}(\theta)=0.
\end{equation}
By the monotone convergence theorem, we have
\begin{eqnarray*}
  \lim_{\theta_{j} \downarrow -\infty} \varphi_{i}(\theta) = \dd{E}_{\pi} \left[ \int_0^1  e^{\br{\theta^{J \setminus \{i,j\}}, Z^{J \setminus \{i,j\}}(t)}} 1(Z_{j}(t) = 0) dY_{i}(t)\right].
\end{eqnarray*}
By (\ref{eq:RBM4}), $Y_i(t)=\int_0^t 1(Z_{i}(s) = 0) dY_{i}(s)$ for all $t\ge
0$. Thus,
\begin{eqnarray*}
\lefteqn{\int_0^1  e^{\br{\theta^{J \setminus \{i,j\}}, Z^{J \setminus
          \{i,j\}}(t)}} 1(Z_{j}(t) = 0) dY_{i}(t)} \\
&& =
 \int_0^1  e^{\br{\theta^{J \setminus \{i,j\}}, Z^{J \setminus
          \{i,j\}}(t)}} 1(Z_{j}(t)= 0, Z_i(t)=0) dY_{i}(t),
\end{eqnarray*}
which equals to zero almost surely by Lemma 4.5 of \cite{DaiWill1995} (see also Theorem 1 of \cite{ReimWill1988}).
Therefore, we have proved (\ref{eq:thetajlimit}).
Hence, for each $j \in J$ and $\theta \le 0$, dividing both sides of \eq{key 1} by $\theta_{j}$ and letting $\theta_{j} \downarrow -\infty$, we have
\begin{eqnarray}
\label{eq:Palm 2}
 - \lim_{\theta_{j} \downarrow -\infty} \left(\frac 12 \Sigma_{jj} \theta_{j} + \mu_{j}\right) \varphi(\theta) &=& R_{jj} \varphi_{j}(\theta) + \sum_{i \ne j} R_{ji} \lim_{\theta_{j} \downarrow -\infty} \varphi_{i}(\theta) \nonumber\\
  &=& R_{jj}\varphi_{j}(\theta).
\end{eqnarray}
Let $\theta = \lift{\theta^K}$ for $j \in K$ in this equation, we have
\begin{eqnarray}
\label{eq:Palm 3}
  - \lim_{\theta_{j} \downarrow -\infty} \left(\frac 12 \Sigma_{jj} \theta_{j} + \mu_{j}\right) \varphi^{K}(\theta^{K}) = R_{jj}\varphi_{j}^{K}(\theta^{K}).
\end{eqnarray}
By the independence assumption, we can write
\begin{eqnarray*}
  \varphi(\theta) = \varphi^{K}(\theta^{K}) \varphi^{L}(\theta^{L}).
\end{eqnarray*}
Hence, multiplying $\varphi^{L}(\theta^{L})$ to both sides of \eq{Palm 3}, then \eq{Palm 2} yields \eq{Palm 1} because $R$ is an completely-$\sr{S}$ matrix and therefore $R_{jj} \ne 0$.
\end{proof}

\subsection{Proof of \thr{marginals}}
\label{sect:theorem 1}
Because the SRBM $Z$ has a stationary distribution,
(\ref{eq:stability}) is satisfied. Thus, $Q \equiv R^{-1}$ exists.
Let $V(t)=QZ(t)$ for $t\ge 0$.
It
follows from \eq{RBM1} that 
\begin{eqnarray}
\label{eq:QZ 1}
  V(t) = V(0) + Q X(t) + Y(t), \qquad t \ge 0.
\end{eqnarray}
Note that $QX$ is still Brownian motion with drift vector $Q\mu$ and
covariance matrix $Q\Sigma Q^T$. For an SRBM $Z$ arising from
open multiclass queueing networks, the process $V$ is known as 
the \emph{total} workload process \cite{HarrNguy1990}. The key idea of
our proof is to use \eq{QZ 1} instead of \eq{RBM1}, and \eq{QZ 1} allows us to
easily separate the entries of $Y(t)$ among coordinates in a 
partition $K$ and $L$. 

The most difficult part is in the proof of part  (a) of the
theorem because there is no hope to directly verify that
   $\tilde{R}(U)$ is completely-$\cal S$ as noted
  in \rem{marginals 0} immediately below the statement of the
    theorem. To verify it, we need information from the
  decomposability.
We  prove  
the first half of part (a) in the 
following lemma.

\begin{lemma} {\rm
\label{lem:invertible}
  $Q^{(K,K)}$ is invertible.
}\end{lemma}
\begin{proof}
  In this proof, we apply truncation arguments similarly to those in \cite{DaiMiya2013}. To this end, we introduce the following sequences of functions. For each positive integer $n$, let
\begin{eqnarray*}
  g_{n}(s) =
  \begin{cases}
  \frac 12 (s+n+2)^{2}, \quad & - (n + 2) < s \le -(n+1),\\
  1 - \frac 12 (s+n)^{2}, \quad & -(n+1) < s \le -n,\\
  1, \quad & -n < s \le n,\\
  1 - \frac 12 (s-n)^{2}, \quad & n < s \le n+ 1,\\
  \frac 12 (n+2 - s)^{2}, \quad & n + 1 < s \le n+2,\\
  0, & s \le -(n+2) \mbox{ or } s > n+2,
  \end{cases}
\end{eqnarray*}
and let
\begin{eqnarray*}
  f_{n}(u) = \left\{\begin{array}{ll}
    \int_{0}^{u} g_{n}(s) ds, \quad & u \ge 0,\\
    - \int_{u}^{0} g_{n}(s) ds, \quad & u < 0.
      \end{array} \right.
\end{eqnarray*}
Clearly, for each fixed $n$, $f_{n}(u)$ is bounded, twice continuously
differentiable, and its derivatives $f'_{n}(u)$ and
$f''_{n}(u)$ are bounded by 1 in absolute values. Furthermore, for
each $u\in \dd{R}$,  $f'_{n}(u) = g_n(u)$ is 
monotone in $n$, and for each $u\in \dd{R}$,
\begin{eqnarray}
\label{eq:fn infty}
\lim_{n\to\infty} f_n(u) = u, \quad \lim_{n\to\infty} f'_n(u) = 1,
\quad \lim_{n\to\infty} f''_n(u) = 0. 
\end{eqnarray}
For each $i \in J$ and $t \ge 0$, we apply It\^{o}'s integration
formula for $f_{n}([QZ(t)]_{i})$ for each
fixed $n$, then it follows from \eq{QZ 1} that
\begin{eqnarray}
\label{eq:Ito 1}
  \lefteqn{\hspace{-7ex} f_{n}([V(t)]_{i}) - f_{n}([V(0)]_{i}) = \int_{0}^{t} f'_{n}([V(u)]_{i}) d([QX(u)]_{i} + Y_{i}(u))} \nonumber\\
 && \hspace{20ex} + \frac 12 \int_{0}^{t} f''_{n}([V(u)]_{i}) [Q \Sigma Q^{\rs{T}}]_{ii} du.
\end{eqnarray}
Because $f_n$, $f_n'$ and $f''_{n}$ are all bounded, we take the
expectation $\dd{E}_\pi$ on both sides of (\ref{eq:Ito 1}) with
$t=1$, and  obtain
\begin{eqnarray*}
  \lefteqn{0 = \int_{\dd{R}^d_+} f'_n([Qx]_i)  [Q\mu]_i \pi(dx) + \int_{\dd{R}^d_+} f'_n([Qx]_i)  \nu_i(dx)} \hspace{10ex}\\
  && \hspace{10ex} +  \frac{1}{2} [Q\Sigma Q^{\rs{t}}]_{ii} 
\int_{\dd{R}^d_+} f''_n([Qx]_i) \pi(dx).
\end{eqnarray*}
Applying the  dominated convergence
theorem on the $f''_n$ term and the monotone convergence theorem 
on two $f'_n$ terms, by letting $n \to \infty$, we have 
\begin{eqnarray}
\label{eq:QK Y}
  (Q\mu)_{i} + \nu_i(\dd{R}^d_+) = 0 \quad i \in J.
\end{eqnarray}

We now assume that $Q^{(K,K)}$ is singular. Then, there exists a
non-zero $|K|$-dimensional row vector $\eta$ such that 
\begin{eqnarray}
\label{eq:singular eta}
  \eta Q^{(K,K)} = 0.
\end{eqnarray}
We will prove that (\ref{eq:singular eta}) implies 
\begin{equation}
  \label{eq:QKL}
\eta Q^{(K, L)} =0.
\end{equation}

Assuming (\ref{eq:QKL}), we now show that it leads to a contradiction,
thus proving the lemma. To see this,  it follows from (\ref{eq:QZ 1})
that 
\begin{equation}
  \label{eq:VK}
V^K(t)= V^K(0) + (QX)^K(t) + Y^K(t).
\end{equation}
Since $V(t)=QZ(t)$, (\ref{eq:singular eta}) and (\ref{eq:QKL}) imply
that $\eta 
V^K(t)=0$ for all $t\ge 0$. Similarly, 
(\ref{eq:singular eta}) and (\ref{eq:QKL}) imply that $\eta
(QX(t))^K=0$ for all $t\ge 0$. Hence, \eq{VK} yields $\eta Y^K(t)=0$ for all $t\ge 0$.
Namely,
\begin{eqnarray*}
  \sum_{i \in K} \eta_{i} Y_{i}^{K}(t) = 0, \qquad t \ge 0.
\end{eqnarray*}
Assume $\eta_j\neq 0$ for some $j\in K$. Since $Y^{K}_{i}(t) = Y_{i}(t)$, we have 
\begin{eqnarray*}
  \sum_{i \in K} \eta_{i} \int_{0}^{t} 1(Z_{j}(u)=0) dY_{i}(u) = 0.
\end{eqnarray*}
By Lemma 4.5 of \cite{DaiWill1995}, we have
\begin{eqnarray*}
\int_{0}^{1} 1(Z_{j}(u)=0) dY_{i}(t)= 0  \quad \text{almost surely for
each pair } i\neq j.
\end{eqnarray*}
This yields
\begin{eqnarray*}
\eta_j\int_{0}^{1} 1(Z_{j}(u)=0) dY_{j}(t)= 0 \quad \text{almost surely},
\end{eqnarray*}
which leads to a contradiction because $\eta_j\neq 0$ for some $j \in J$ and \eq{QK Y} together with \eq{stability} implies, for every $j \in J$,
\begin{eqnarray*}
\dd{E}_\pi\left(\int_{0}^{1} 1(Z_{j}(u)=0)  dY_{j}(t)\right) =
\nu_j(\dd{R}^d_+) = -(Q\mu)_j>0.  
\end{eqnarray*}

Now we prove (\ref{eq:QKL}).
Note that $(QX)^K$ in (\ref{eq:VK}) is a $\abs{K}$-dimensional
Brownian motion with drift 
$Q^{(K, J)}\mu$ and covariance matrix 
\begin{displaymath}
Q^{(K, J)}\Sigma
(Q^{\rs{t}})^{(J, K)} =(Q\Sigma Q^{\rs{t}})^{(K, K)}.
\end{displaymath}
We apply It\^{o}'s integration formula to $h_n(V^K(t))$, where
$h_{n}(x) = e^{-f_{n}(\br{\eta,x})}$ for $x \in \dd{R}^{|K|}$. Then,
we have 
\begin{eqnarray}
\label{eq:Ito 2}
  \lefteqn{h_{n}(V^K(t)) - h_{n}(V^K(0)) = \sum_{i \in K} \int_{0}^{t}  \left. \frac {\partial h_{n}(x)}
    {\partial x_{i}} \right|_{x= V^K(u)} d(QX)^K_i(u)} \hspace{25ex}\nonumber\\ 
  && + \frac 12  \sum_{i,j \in K} \int_{0}^{t} \left. \frac {\partial^{2} h_{n}(x)} {\partial x_{i} \partial x_{j}} \right|_{x=V^{K}(u)} (Q \Sigma Q^{\rs{t}})^{(K,K)}_{ij} du \nonumber\\ 
  && + \sum_{i \in K} \int_{0}^{t}  \left. \frac {\partial h_{n}(x)} {\partial x_{i}} \right|_{x= V^{K}(u)} dY_{i}(u),
\end{eqnarray}
where 
\begin{eqnarray*}
 && \frac {\partial h_{n}(x)} {\partial x_{i}} = - \eta_{i}
 g_{n}(\br{\eta,x}) h_{n}(x), \\
 && \frac {\partial^{2} h_{n}(x)} {\partial x_{i} \partial x_{j}} = \eta_{i} \eta_{j} (g^{2}_{n}(\br{\eta,x}) + g'_{n}(\br{\eta,x})) h_{n}(x).
\end{eqnarray*}
Setting $t=1$ and taking expectation $\dd{E}_\pi$ on both side, 
we have
\begin{eqnarray}
 \lefteqn{0  = \sum_{i \in K} \left. \dd{E}_\pi\left[\frac {\partial h_{n}(x)}
    {\partial x_{i}} \right|_{x= V^K(0)} \right] (Q\mu)_i} \hspace{5ex} \nonumber\\
   && + \frac 12  \sum_{i,j \in K} (Q \Sigma Q^{\rs{t}})^{(K,K)}_{ij} \left. \dd{E}_\pi\left[\frac 
    {\partial^{2} h_{n}(x)} {\partial x_{i} \partial x_{j}}
  \right|_{x=V^{K}(0)} \right] \nonumber\\ 
  && {} + \sum_{i \in K} \dd{E}_{\pi}\left[\int_{0}^{1}  \left. \frac {\partial
      h_{n}(x)} {\partial x_{i}} \right|_{x= V^{K}(u)} dY_{i}(u)\right].  \label{eq:ito3}
\end{eqnarray}
Recall that $V^K(t)=Q^{(K,K)}Z^K(t)+Q^{K, L}Z^{L}(t)$. Let
\begin{displaymath}
M^K(t)=  Q^{K, L} Z^L(t).
\end{displaymath}
Because $\eta
Q^{(K,K)}=0$, we have $\eta V^K(t)=\eta M^K(t)$, and therefore, for $i\in K$,
\begin{eqnarray}
\label{eq:boundary2}
 \lefteqn{\hspace{-9ex} \dd{E}_{\pi}\left[\int_{0}^{1}  \left. \frac {\partial
 h_{n}(x)} {\partial x_{i}} \right|_{x= V^{K}(u)}
  dY_{i}(u)\right] = \dd{E}_{\pi}\left[\int_{0}^{1}  \left. \frac {\partial
      h_{n}(x)} {\partial x_{i}} \right|_{x= M^{K}(u)} dY_{i}(u)\right]} \hspace{20ex}\nonumber\\
&& = \nu_i(\dd{R}^d_+) 
\dd{E}_{\pi}\left[ \left. \frac {\partial
      h_{n}(x)} {\partial x_{i}} \right|_{x= M^{K}(0)} \right],
\end{eqnarray}
where in the second equality, we have used that fact that 
$M^{K}(t)$ is a functions of $Z^{L}(t)$ and \lem{partial independence}.
It follows from (\ref{eq:QK Y}), (\ref{eq:ito3}), and
(\ref{eq:boundary2}) that 
\begin{displaymath}
  \frac 12  \sum_{i,j \in K} (Q \Sigma Q^{\rs{t}})^{(K,K)}_{ij}
  \left. \dd{E}_\pi\left[\frac 
    {\partial^{2} h_{n}(x)} {\partial x_{i} \partial x_{j}}
  \right|_{x=V^{K}(0)} \right]=0,
\end{displaymath}
or equivalently
\begin{eqnarray}
\label{eqn:eta Q}
 \lefteqn{\frac 12 \sum_{i,j \in K} \eta_{i} \eta_{j} (Q \Sigma Q^{\rs{t}})^{(K,K)}_{ij}} \hspace{2ex} \nonumber\\
 && \times \dd{E}_{\pi}((g^{2}_{n}(\br{\eta,V^{K}(0)}) +
  g'_{n}(\br{\eta,V^{K}(0)})) h_{n}(V^{K}(0))) =0.
  \end{eqnarray}
By the construction of functions $g_n$ and $f_n$, $g'_{n}(u) \ge 0$
except for $u 
\in (n,n+2)$, in which 
 $g'_{n}(u) \in [-1,0)$, and 
$e^{-f_{n}(u)}$ monotonically converges
 to $e^{-u}$ as $n \to \infty$ and is bounded by $1$ for $u \ge 0$.
 Furthermore, $g_n(u)$ and $g_n'(u)$ are bounded by
 $1$  for all $u$ and $n$. We have 
\begin{eqnarray}
\label{eqn:gn inequality}
  \dd{E}_{\pi}\Bigl[(g^{2}_{n}(\br{\eta,V^{K}(0)}) +
  g'_{n}(\br{\eta,V^{K}(0)})) 1_{( \br{\eta,V^{K}(0)}\le 0)}
  h_{n}(V^{K}(0))\Bigr] \ge  0
\end{eqnarray} 
for each $n\ge 1$.
By dominated converge theorem,
\begin{eqnarray*}
\lefteqn{\lim_{n\to\infty } \dd{E}_{\pi}\Bigl[(g^{2}_{n}(\br{\eta,V^{K}(0)}) +
  g'_{n}(\br{\eta,V^{K}(0)})) 1_{( \br{\eta,V^{K}(0)}> 0)}
  h_{n}(V^{K}(0))\Bigr]} \hspace{30ex}\\
&& = \dd{E}_{\pi}\Bigl[  1_{( \br{\eta,V^{K}(0)}> 0)}
e^{-\br{\eta,V^{K}(0)}}\Bigr]>0, 
\end{eqnarray*} 
where the strict inequality follows from the fact that the Lebesgue measure of 
 set $\{z\in \dd{R}^d_+: \eta z^K>0\}$ is positive and the fact that
 $\pi(A)>0$ for every measurable set $A$ that has positive Lebesgue
 measure \cite{DaiHarr1992}. Therefore, \eqn{gn inequality} can be sharpened for some large $n_{0}$ in such a way that, for every $n \ge n_{0}$,
\begin{eqnarray*}
\dd{E}_{\pi}\Bigl[(g^{2}_{n}(\br{\eta,V^{K}(0)}) +
  g'_{n}(\br{\eta,V^{K}(0)})) 1_{( \br{\eta,V^{K}(0)}\le 0)}
  h_{n}(V^{K}(0))\Bigr]>  0.
\end{eqnarray*} 
Thus, from \eqn{eta Q} we arrive at
\begin{eqnarray*}
  \frac 12 \sum_{i,j \in K} \eta_{i} \eta_{j} (Q \Sigma Q^{\rs{t}})^{(K,K)}_{ij} = 0.
\end{eqnarray*}
Namely,
\begin{eqnarray*}
\frac{1}{2}\eta Q^{(K, J)}  \Sigma (Q^{\rs{t}})^{(J, K)} \eta^{\rs{t}}=0.
\end{eqnarray*}
Since $\Sigma$ is positive definite, for this to be true, we must have 
\begin{eqnarray}
\label{eq:hQ 1}
\eta Q^{(K, J)}=0,
\end{eqnarray}
thus proving (\ref{eq:QKL}). 
\end{proof}
We now return to the proof of \thr{marginals}. Because 
$V^K(t)=Q^{(K,K)}Z^K(t)+Q^{K,L}Z^L(t)$
and $Q^{(K,K)}$ is
invertible by \lem{invertible}, we have 
\begin{eqnarray}
\lefteqn{Z^K(t) + W^K(t) =  (Q^{(K,K)})^{-1} V^K(t) }  \nonumber\\
&& = (Q^{(K,K)})^{-1} V^K(0) +  (Q^{(K,K)})^{-1} (QX)^K(t) +  (Q^{(K,K)})^{-1} Y^K(t), \hspace{7ex}
\label{eq:Z 2}
\end{eqnarray}
where 
\begin{eqnarray*}
  W^{K}(t) = (Q^{(K,K)})^{-1} Q^{(K,L)} Z^{L}(t),
\end{eqnarray*}
and the second equality follows from (\ref{eq:VK}).
Note that  
\begin{displaymath}
(Q^{(K,K)})^{-1} (QX)^K(t)
\end{displaymath}
 is a  
$\abs{K}$-dimensional Brownian motion with drift vector
$ \tilde{\mu}(K) = (Q^{(K,K)})^{-1} (Q \mu)^{K}$ 
and covariance
matrix
\begin{eqnarray*}
 \tilde{\Sigma}(K) = (Q^{(K,K)})^{-1} (Q \Sigma Q^{\rs{t}})^{(K,K)}
 ((Q^{(K,K)})^{-1})^{\rs{t}}.   
\end{eqnarray*}
We now apply It\^{o}'s integral formula to $f(Z^{K}(t) + W^{K}(t))$,
where   $f(x) \equiv e^{ \imath \br{\theta^{K}, x}}$ with $x \in \dd{R}^{|K|}$ for each
fixed $\theta^{K} \in \dd{R}^{|K|}$ and  $\imath =
\sqrt{-1}$ is the imaginary unit of a complex number. 
(We really apply the It\^{o} formula twice, one for
$\cos(\br{\theta^{K}, x})$  and one for $\sin(\br{\theta^{K}, x})$.)
We have 
\begin{eqnarray*}
  \lefteqn{f(Z^{K}(t) + W^{K}(t)) - f(Z^{K}(0) + W^{K}(0))}\\
  && \hspace{-2ex} = - \frac 12 \int_{0}^{t} \sum_{i,j \in K} \theta_{i} \theta_{j}
  [\tilde \Sigma(K)]_{ij} f(Z^{K}(u) + W^{K}(u))du\\
  && + \imath \int_{0}^{t} \sum_{i \in K} \theta_{i} f(Z^{K}(u)
  + W^{K}(u))  d [(Q^{(K,K)})^{-1} (Q X)^{K}]_{i}\\ 
  && + \imath \int_{0}^{t} \sum_{i \in K} \theta_{i} f(Z^{K}(u) + W^{K}(u)) \sum_{j \in K} [\tilde R(K)]_{ij} dY_{j}(u),
\end{eqnarray*}
where 
\begin{eqnarray*}
 \tilde{R}(K) = (Q^{(K,K)})^{-1}.
\end{eqnarray*}
Because $|f(Z^{K}(u) + W^{K}(u)))| \le 1$, setting $t=1$,
we can take  expectation $\dd{E}_\pi$ on both sides of this equation  for
$\theta \in \dd{R}^{d}$, 
we have 
\begin{eqnarray}
\label{eq:ZK 0}
  \lefteqn{\hspace{-3ex}-\frac 12 \sum_{i,j \in K} \theta_{i} \theta_{j} [\tilde
    \Sigma(K) ]_{ij}
\dd{E}_{\pi}\left(e^{\imath \br{\theta^{K}, Z^{K}(0)}+ \imath \br{\theta^{K}, W^{K}(0)}} \right)} \nonumber \\
  && \hspace{-5ex} + \imath \sum_{i \in K} \theta_{i} [\tilde \mu(K)]_{i} \dd{E}_{\pi} \left(e^{\imath \br{\theta^{K}, Z^{K}(0)}+ \imath \br{\theta^{K}, W^{K}(0)}} \right) \nonumber\\
  && \hspace{-5ex} + \imath \sum_{i,j \in K} \theta_{i} [\tilde
  R(K)]_{ij} \dd{E}_{\pi}\left(\int_0^1 e^{\imath \br{\theta^{K}, Z^{K}(u)}+
      \imath \br{\theta^{K}, W^{K}(u)}} dY_j(u)\right) = 0. 
\end{eqnarray}

  We now use the assumption that $Z^{K}(0)$ and $Z^{L}(0)$ are
  independent under $\pi$. Under the independence assumption, one can
  show that \eq{Palm 1} also holds when $\theta$ is
  replaced by $\imath \theta$, i.e., 
  \begin{equation}
    \label{eq:PalmComplex}
    \varphi_j(\imath \theta) =     \varphi_j^K(\imath \theta^K)\varphi^L(\imath \theta^L) \quad \text{ for each } \theta\in \R^d.
  \end{equation}
  Equation (\ref{eq:PalmComplex}) can be proved following the proof of
  Lemma \ref{lem:partial independence} by using Riemann-Lebesgue lemma (e.g., see Section 7.1 of \cite{Stri2003}) for \eq{thetajlimit} and the characteristic
  function version of \eq{key 1} when $\theta$ is replaced by $\imath
  \theta$. For each $\theta^{K} \in \dd{R}^{|K|}$, let $\theta\in\R^d$
  be the corresponding vector with $\theta^{L} = \theta^{K} (Q^{(K,K)})^{-1}
  Q^{(K,L)}$.  For this $\theta\in \R^d$, it follows from the
  definition of $W^{K}(t)$ that (\ref{eq:PalmComplex}) is equivalent to the following equality
\begin{align*}
\dd{E}_{\pi} & \left(\int_0^1 e^{\imath \br{\theta^{K}, Z^{K}(u)}+
      \imath \br{\theta^{K}, W^{K}(u)}} dY_j(u)\right)\\
    & \hspace{10ex} =
\dd{E}_{\pi}\left(\int_0^1 e^{\imath \br{\theta^{K}, Z^{K}(u)}} dY_j(u)\right) 
\dd{E}_{\pi} \left(    e^{ \imath \br{\theta^{K}, W^{K}(0)}} \right).
\end{align*}
Also, $Z^{K}(0)$ and $W^{K}(0)$ are independent because $W^{K}(0)$ is a linear transform of $Z^{L}(0)$. Thus, combining these facts with \eq{ZK 0} yields
\begin{eqnarray}
\label{eq:ZK 1}
  \lefteqn{\Bigl(-\frac 12 \sum_{i,j \in K} \theta_{i} \theta_{j} [\tilde
    \Sigma(K) ]_{ij} 
+ \imath \sum_{i \in K} \theta_{i} [\tilde \mu(K)]_{i}\Bigr)
\dd{E}_{\pi}\bigl(e^{\imath \br{\theta^{K}, Z^{K}(0)}} \bigr)} \hspace{5ex}\nonumber \\
  && {} + \imath \sum_{i,j \in K} \theta_{i} [\tilde
  R(K)]_{ij} 
 \dd{E}_{\pi}\left(\int_0^1 e^{\imath \br{\theta^{K}, Z^{K}(u)}}
 dY_j(u)\right) = 0. 
\end{eqnarray}
In what follows we will get the moment generating function version of \eq{ZK 1} with $\theta^{K} \le 0$. One may wonder why we do not directly consider this generating function version. This is because $\theta^{L} \equiv \theta^{K} (Q^{(K,K)})^{-1} Q^{(K,L)} \le 0$ may not be true in \eq{Palm 1} for $\theta^{K} \le 0$.

Denote the left-hand side of \eq{ZK 1} by $g(\imath \theta)$ as a
function of $\imath \theta$ for $\theta \in \dd{R}^{|K|}$. We can replace
$\imath \theta$ of $g(\imath \theta)$ by a complex vector $z \equiv (z_{1}, z_{2}, \ldots,
z_{|k|})$ such that $\Re z_{i} \le 0$ for all $i =1, \ldots, K$, where $\Re
z_{i}$ is the real part of $z_{i}$. Then, it is easy to see that $g(z_{1}, z_{2},
\ldots, z_{|K|})$ is analytic in each $z_{i}$ such that $\Re z_{i} <
0$ when  $z_{j}$ for $j \not= i$ is fixed satisfying $\Re z_{j} \le 0$. Let $i=1$ and fix an arbitrary $\theta \in \dd{R}^{|K|}$. Since $g(z_{1}, \imath \theta_{2}, \ldots, \imath \theta_{|K|})$ converges to $g(\imath \theta) \equiv 0$ as $z_{1}$ with $\Re z_{1} < 0$ continuously moves to $\imath \theta_{1}$, we must have
\begin{eqnarray*}
  g(z_{1}, \imath \theta_{2}, \ldots, \imath \theta_{|K|}) = 0
     \text{ for } \Re z_{1} \le 0
\end{eqnarray*}
by the so called boundary uniqueness theorem (e.g., see page 371 of Volume I of \cite{Mark1977}). We then inductively replace $\imath \theta_{i}$ by $z_{i}$ with $\Re z_{i} \le 0$ for $i=2,3, \ldots, |K|$, and we have
\begin{eqnarray*}
  g(z_{1}, z_{2}, \ldots, z_{|K|}) = 0, \quad \Re z_{i} \le 0 \mbox{ for } i=1,2,\ldots,|K|.
\end{eqnarray*}
In particular, letting $z_{i} = \theta_{i}$ for real $\theta_{i} \le 0$ for $i=1,2,\ldots, |K|$, we have
\begin{eqnarray}
\label{eq:ZK 2}
 \lefteqn{\left( \frac 12 \sum_{i,j \in K} \theta_{i} \theta_{j} [\tilde{\Sigma}(K)]_{ij} + \sum_{i \in K} \theta_{i} [\tilde{\mu}(K)]_{i} \right) \varphi^{K}(\theta^{K})} \nonumber \hspace{20ex}\\
 && + \sum_{i,j \in K} \theta_{i} [\tilde{R}(K)]_{ij} \varphi_{j}^{K}(\theta^{K}) = 0.
\end{eqnarray}
We are now ready to prove the remaining part of Theorem \ref{thr:marginals}, part (a), in the following lemma.

\begin{lemma}
\label{lem:complete-S}
  Under the assumptions of \thr{marginals}, $\tilde{R}(K)$ and $\tilde{R}(L)$ are completely-$\sr{S}$ matrices.
\end{lemma}
\begin{proof}
  For an arbitrarily fixed $\ell \in K$, let $\theta_{j}^{K} = 0$ in $\theta^{K}$ for $j \ne \ell$. We denote this vector $\theta^{K}$ as $(\lift{\theta_{\ell}})^{K}$. For $j \ne \ell$, let $\theta_{j} = 0$ in \eq{ZK 2} and divide the resulting formula by $\theta_{\ell} < 0$, then we have
\begin{eqnarray}
\label{eq:ZK 3}
 \hspace{-3ex} \left(\frac 12 \theta_{\ell} [\tilde{\Sigma}(K)]_{\ell \ell} + [\tilde{\mu}(K)]_{\ell} \right) \varphi^{K}((\lift{\theta_{\ell}})^{K}) + \sum_{j \in K} [\tilde{R}(K)]_{\ell j} \varphi_{j}^{K}((\lift{\theta_{\ell}})^{K}) = 0.
\end{eqnarray}
Similarly to \eq{Palm 2},
\begin{eqnarray*}
  - \lim_{\theta_{\ell} \downarrow -\infty} \Big(\frac 12 \theta_{\ell} [\tilde{\Sigma}(K)]_{\ell \ell} + [\tilde{\mu}(K)]_{\ell} \Big) \varphi^{K}((\lift{\theta_{\ell}})^{K}) = [\tilde{R}(K)]_{\ell \ell} \varphi_{\ell}^{K}(0^{K}).
\end{eqnarray*}
Since the left-hand side of this formula is positive by \eq{Palm 2},
its right-hand side must be positive. Hence, $[\tilde{R}(K)]_{\ell
  \ell} > 0$. Furthermore, we can take sufficiently small
$\theta_{\ell} < 0$ such that 
\begin{eqnarray*}
  \left(\frac 12 \theta_{\ell} [\tilde{\Sigma}(K)]_{\ell \ell} + [\tilde{\mu}(K)]_{\ell} \right) \varphi^{K}((\lift{\theta_{\ell}})^{K}) < 0.
\end{eqnarray*}
By \eq{ZK 3}, we have, for this $\theta_{\ell}$,
\begin{eqnarray}
\label{eq:positive 1}
  \sum_{j \in K} [\tilde{R}(K)]_{\ell j} \varphi_{j}^{K}((\lift{\theta_{\ell}})^{K}) > 0.
\end{eqnarray}
Since $\varphi_{j}^{K}((\lift{\theta_{\ell}})^{K}) > 0$ for all $j \in
K$, $\tilde{R}(K)$ is an $\sr{S}$-matrix. Let $U$ be a subset of $K$
such that $\ell \in U$ and $U \ne K$, then we can choose
$\theta_{\ell}$ such that
$\varphi_{j}^{K}((\lift{\theta_{\ell}})^{K})$ is sufficiently small
for $j \in K \setminus U$. This yields that $(\tilde{R}(K))^{(U,U)}$
is an $\sr{S}$-matrix, and therefore we have proved that $\tilde{R}(K)$
is a completely-$\sr{S}$ matrix. 
\end{proof}

We now can see that \eq{ZK 1} is nothing but the
 BAR for the $|K|$-dimensional SRBM with data
$(\tilde{\Sigma}(K), \tilde{\mu}(K), \tilde{R}(K))$ because
$\tilde{R}(K)$ is completely-$\sr{S}$ by \lem{complete-S}. By part (b)
of Lemma \ref{lem:key 1}, the $\abs{K}$-dimensional
$(\tilde{\Sigma}(K), \tilde{\mu}(K), \tilde{R}(K))$-SRBM has a
stationary distribution that is equal to the distribution of $Z^K(0)$
under $\pi$. By the symmetric roles of $K$ and $L$, the 
the $\abs{L}$-dimensional
$(\tilde{\Sigma}(L), \tilde{\mu}(L), \tilde{R}(L))$-SRBM has a
stationary distribution that is equal to the distribution of $Z^L(0)$
under $\pi$.

\subsection{Proof of \thr{decomposition}}
\label{sect:theorem 2}

First of all, we will prove if $Z^{K}(0)$ and $Z^{L}(0)$ are
independent and $Z^{K}(0)$ is of product form under $\pi$, then
\eq{condition1} and \eq{condition2} hold. According to
\thr{marginals}, the distribution of $Z^{\{i\}}(0)$ under $\pi$ is
equal to the stationary distribution of $(\tilde{\Sigma}(\{i\}),
\tilde{\mu}(\{i\}), \tilde{R}(\{i\}))$-SRBM for $i \in K$. The
distribution of $Z^{L}(0)$ under $\pi$ is equal to the stationary
distribution of $(\tilde{\Sigma}(L), \tilde{\mu}(L),
\tilde{R}(L))$-SRBM. Then by \lem{key 1}, for $\theta\in \dd{R}^d$ with
$\theta\le 0$,
\begin{eqnarray}
\label{eq:1,d-1}
&& -\left(\frac{1}{2}\frac{\Sigma_{ii}}{R_{ii}}\theta_i+((R^{(K,K)})^{-1}\mu^K)_i\right)\varphi^{\{i\}}(\theta_i) =1, \quad i \in K, \\
\label{eq:1,d-2}
&& -\left( \frac 12 \sum_{i,j \in L} \theta_{i} \theta_{j} [\tilde{\Sigma}(L)]_{ij} + \sum_{i \in L} \theta_{i} [\tilde{\mu}(L)]_{i} \right) \varphi^{L}(\theta^{L}) \nonumber \hspace{10ex}\\
&& \hspace{30ex} = \sum_{i,j \in L} \theta_{i} [\tilde{R}(L)]_{ij} \varphi_{j}^{L}(\theta^{L}).
\end{eqnarray}
By the definition of $\tilde{\Sigma}(L)$, $\tilde{\mu}(L)$ and $\tilde{R}(L)$, we can find they are
\begin{eqnarray*}
 && \hspace{-3ex} \tilde{\Sigma}(L)=\Sigma^{(L,L)} +R^{(L,K)}(R^{(K,K)})^{-1}\Sigma^{(K,K)}((R^{(K,K)})^{-1})^{\rs{t}}(R^{(L,K)})^{\rs{t}}\nonumber\\ 
&& \hspace{6ex} -\Sigma^{(L,K)}((R^{(K,K)})^{-1})^{\rs{t}}(R^{(L,K)})^{\rs{t}} - R^{(L,K)}(R^{(K,K)})^{-1}(\Sigma^{(L,K)})^{\rs{t}},\\
 && \hspace{-3ex}  \tilde{\mu}(L)=\mu^{L}-R^{(L,K)}(R^{(K,K)})^{-1}\mu^K, \\
&& \hspace{-3ex} \tilde{R}(L)=R^{(L,L)}.
\end{eqnarray*}  
By the independence assumptions, we have, for $i \in K$,
\begin{eqnarray*}
  \varphi(\theta)=\varphi^{\{i\}}(\theta_i)\varphi^{J\setminus\{i\}}(\theta^{J\setminus\{i\}})=\varphi^{K}(\theta^{K})\varphi^{L}(\theta^{L}).
\end{eqnarray*}
By \lem{partial independence}, we get
\begin{eqnarray*}
  \varphi_i(\theta)=\varphi^{J\setminus\{i\}}(\theta^{J\setminus\{i\}}), \quad i \in K, \qquad
  \varphi_j(\theta)=\varphi^{K}(\theta^{K})\varphi_j^{L}(\theta^{L}), \quad j \in L.
\end{eqnarray*}
 Using the fact that $\varphi_i^{\{i\}}(0)=1$, we can
rewrite \eq{1,d-1} and \eq{1,d-2} as 
\begin{eqnarray}
\label{eq:1}
&& \hspace{-10ex} -\left(\frac{1}{2}\frac{\Sigma_{ii}}{R_{ii}}\theta_i+((R^{(K,K)})^{-1}\mu^K)_i\right)\varphi(\theta) = \varphi_i(\theta), \qquad i \in K,\\
\label{eq:d-1}
&& \hspace{-10ex} -\left( \frac 12 \sum_{i,j \in L} \theta_{i} \theta_{j} [\tilde{\Sigma}(L)]_{ij} + \sum_{i \in L} \theta_{i} [\tilde{\mu}(L)]_{i} \right) \varphi(\theta) = \sum_{i,j \in L} \theta_{i} [\tilde{R}(L)]_{ij} \varphi_{j}(\theta).
\end{eqnarray}
We further modify \eq{1} into
\begin{eqnarray}
\label{eq:1modify}
 -\left(\frac{1}{2}\frac{\Sigma_{ii}}{R_{ii}}\theta_i+((R^{(K,K)})^{-1}\mu^K)_i\right)\gamma_i(\theta)\varphi(\theta) = \gamma_i(\theta)\varphi_i(\theta).
\end{eqnarray}
Then adding \eq{1modify} for $i \in K$ and \eq{d-1}, we have
\begin{eqnarray}
\label{eq:keymodify}
\lefteqn{-\Bigg(\sum_{i \in K}\left(\frac{1}{2}\frac{\Sigma_{ii}}{R_{ii}}\theta_i+((R^{(K,K)})^{-1}\mu^K)_i\right)\gamma_i(\theta)} \nonumber\\
&&+\frac{1}{2}\sum_{i,j \in L} \theta_{i} \theta_{j} [\tilde{\Sigma}(L)]_{ij} + \sum_{i \in L} \theta_{i} [\tilde{\mu}(L)]_{i} \Bigg) \varphi(\theta) =\sum_{i=1}^{d} \gamma_{i}(\theta) \varphi_{i}(\theta).
\end{eqnarray}
So according to \lem{key 1}, we have
\begin{eqnarray}
\label{eq:equalgamma}
\lefteqn{\hspace{-5ex} \gamma(\theta)=-\Bigg(\sum_{i \in K}\left(\frac{1}{2}\frac{\Sigma_{ii}}{R_{ii}}\theta_i+((R^{(K,K)})^{-1}\mu^K)_i\right)\gamma_i(\theta)} \hspace{10ex}\nonumber\\
&& \quad +\frac{1}{2}\sum_{i,j \in L} \theta_{i} \theta_{j} [\tilde{\Sigma}(L)]_{ij} + \sum_{i \in L} \theta_{i} [\tilde{\mu}(L)]_{i} \Bigg).
\end{eqnarray}
Comparing the coefficients of $\theta_1\theta_j$ and coefficients of $\theta_i$ of both sides, we can get
\begin{eqnarray}
\label{eq:condition11}
&&\hspace{-8ex} -\frac{1}{2}(\Sigma_{ij}+\Sigma_{ji})=-\frac{\Sigma_{ii}}{2R_{ii}}R_{ji}-\frac{\Sigma_{jj}}{2R_{jj}}R_{ij}, \; i,j \in K, \\
\label{eq:condition21}
&&\hspace{-8ex} -\frac{1}{2}(\Sigma_{ij}+\Sigma_{ji})=-\frac{\Sigma_{ii}}{2R_{ii}}R_{ji}, \qquad i \in K, \; j \in L,\\
\label{eq:condition3}
&&\hspace{-8ex} R^{(L,K)}(R^{(K,K)})^{-1}\Sigma^{(K,K)}((R^{(K,K)})^{-1})^{\rs{t}}(R^{(L,K)})^{\rs{t}}\nonumber\\ 
&& \hspace{-6ex}- \Sigma^{(L,K)}((R^{(K,K)})^{-1})^{\rs{t}}(R^{(L,K)})^{\rs{t}} - R^{(L,K)}(R^{(K,K)})^{-1}(\Sigma^{(L,K)})^{\rs{t}}=0.
\end{eqnarray}
Observe that \eq{condition11} is equivalent to \eq{condition1}, and
\eq{condition21} is equivalent to \eq{condition2}. Under
\eq{condition1} and \eq{condition2}, \eq{condition3} automatically
holds. So we have proved if $Z^{K}(0)$ and $Z^{L}(0)$ are independent and
$Z^{K}(0)$ is of product form under $\pi$, then \eq{condition1} and
\eq{condition2} hold. 

Next we will prove that if $\Sigma$ and $R$ satisfy (\ref{eq:condition1}) and
  (\ref{eq:condition2}), 
  and the $\abs{L}$-dimensional $(\Sigma^{(L,L)}, \tilde{\mu}(L),
  R^{(L,L)})$-SRBM has a stationary distribution, then $Z^{K}(0)$ and
  $Z^{L}(0)$ are independent and $Z^{K}(0)$ is of product form under
  $\pi$. First observe that if \eq{condition1} and \eq{condition2}
  holds, then \eq{condition11}, \eq{condition21} and \eq{condition3}
  hold. Therefore, \eq{equalgamma} holds.  

Because the skew symmetry condition (\ref{eq:condition1}) holds, the
$(\Sigma^{(K,K)}$, $\mu^{(K,K)}$, 
$R^{(K,K)})$-SRBM has a product form stationary distribution
\cite{HarrWill1987a}. Let 
$\tilde{\varphi}^K(\theta^K)$ and 
$\tilde{\varphi}_i^K(\theta^K)$ be the moment generating functions of
the stationary distribution and $i$th boundary measure for this
$(\tilde{\Sigma}(K), \tilde{\mu}(K), \tilde{R}(K))$-SRBM. Then  
\begin{eqnarray*}
&& \hspace{-4ex} -\left(\frac{1}{2}\frac{\Sigma_{ii}}{R_{ii}}\theta_i+((R^{(K,K)})^{-1}\mu^K)_i\right)(\tilde{\varphi}^K)^{\{i\}}(\theta_i) =1, \quad i \in K.
\end{eqnarray*}
Because the $\abs{L}$-dimensional $(\Sigma^{(L,L)}, \tilde{\mu}(L),
  R^{(L,L)})$-SRBM has a stationary distribution,
  $\tilde{\Sigma}(L)=\Sigma^{(L,L)}$ and $\tilde{R}(L)=R^{(L,L)}$, we have
\begin{eqnarray*}
  -\left( \frac 12 \sum_{i,j \in L} \theta_{i} \theta_{j} [\tilde{\Sigma}(L)]_{ij} + \sum_{i \in L} \theta_{i} [\tilde{\mu}(L)]_{i} \right) \tilde{\varphi}^{L}(\theta^{L}) = \sum_{i,j \in L} \theta_{i} [\tilde{R}(L)]_{ij} \tilde{\varphi}_{j}^{L}(\theta^{L}).
\end{eqnarray*}
where $\tilde{\varphi}^{L}(\theta^{L})$ and
$\tilde{\varphi}_{j}^{L}(\theta^{L})$ are the  moment generating
functions of the stationary distribution and $j$th boundary measure for
$(\Sigma^{(L,L)}, \tilde{\mu}(L), 
  R^{(L,L)})$-SRBM.

Let $\tilde{\varphi}(\theta)=\tilde{\varphi}^{K}(\theta^{K})\tilde{\varphi}^{L}(\theta^{L})$, $\tilde{\varphi}_{i}(\theta)=\tilde{\varphi}^K_{i}(\theta^K)\tilde{\varphi}^L(\theta^L)$ for $i \in K$ and $\tilde{\varphi}_{j}(\theta)=\tilde{\varphi}^K(\theta^K)\tilde{\varphi}^L_{j}(\theta^L)$ for $j \in L$. Then we can see \eq{1}, \eq{d-1}, \eq{1modify} and \eq{keymodify} hold with $\varphi(\theta)$, $\varphi_i(\theta)$ and $\varphi_j(\theta)$ replaced by $\tilde{\varphi}(\theta)$, $\tilde{\varphi}_i(\theta)$ and $\tilde{\varphi}_j(\theta)$. Furthermore, as \eq{equalgamma} holds, we conclude
\begin{equation*}
\gamma(\theta)\tilde{\varphi}(\theta)=\sum_{i=1}^{d}\gamma_i(\theta)\tilde{\varphi}_{i}(\theta)
\end{equation*}
By part (b) of  \lem{key 1}, we know
$\varphi(\theta)=\tilde{\varphi}(\theta)$,
$\varphi^K(\theta^K)=\tilde{\varphi}^K(\theta^K)$ and
$\varphi^L(\theta^L)=\tilde{\varphi}^L(\theta^L)$. So
$\varphi(\theta)=\varphi^K(\theta^K)\varphi^L(\theta^L)$, that is,
$Z^{K}(0)$ and $Z^{L}(0)$ are independent under $\pi$. Furthermore,
the distribution of $Z^{K}(0)$ under $\pi$ is equal to the stationary
distribution of the $|{K}|$-dimensional $(\Sigma^{(K,K)}$, $\mu^{K}$,
$R^{(K,K)})$-SRBM. By \eq{condition1}, the distribution of $Z^{K}(0)$
is of product form  because the  skew symmetry condition (\ref{eq:skew 
  symmetric})  in \cite{HarrWill1987a} is satisfied. 

\section{Concluding remarks}
\label{sect:concluding}

There are two directions for future study. We first
comment on the marginal distributions. In the proof of
\thr{marginals}, we may only use the following fact to complete the
proof. Random vectors
 $Z^{K}(0)$ and $W^{K}(0) \equiv (Q^{(K,K)})^{-1} Q^{(K,L)}
Z^{L}(0)$ are ``weakly independent through convolution'' under the
stationary distribution $\pi$, that is, for all $\theta \in
\dd{R}^{|K|}$, 
\begin{eqnarray}
\label{eq:convolution indep}
  \dd{E}_{\pi}(e^{\imath \br{\theta, (Z^{K}(0) + W^{K}(0))}}) = \dd{E}_{\pi}(e^{\imath \br{\theta, Z^{K}(0)}}) \dd{E}_{\pi}(e^{\imath \br{\theta, W^{K}(0)}}),
\end{eqnarray}
where $\imath = \sqrt{-1}$ is again the imaginary unit. 

In general, we can easily find an example such that random
  variables $X$ and $Y$ are not independent but weakly independent
  through convolution. However, this may not be the case for
  $Z^{K}(0)$ and $W^{K}(0)$ because they have more structure. To
  further discuss this issue, let us consider the tandem queue example
  of \sectn{introduction} for $d=2$, $K=\{1\}$ and $L=\{2\}$. Since
  $Q^{(K,L)} = \{0\}$ and $Q^{(L,K)} = \{1\}$, we have $W^{K}(t) = 0$
  and $W^{L}(t) = Z_{1}(t)$. Hence, \eq{convolution indep} for $L$
  instead of $K$ is equivalent to that $Z_{1}(0)$ and $Z_{2}(0)$ are
  weakly independent through convolution under $\pi$. We consider the
  implication of this observation below.

We first observe that \eq{alpha} and \cor{marginal 1} yield that
\begin{eqnarray*}
  \alpha_{1} = \lambda_{1} = \frac {2(\beta_{1} - \beta_{0})} {\beta_{0} (c^{2}_{0} + c^{2}_{1})}, \quad \alpha_{2} = \frac {2(\beta_{2} - \beta_{0})} {\beta_{0} (c^{2}_{1} + c^{2}_{2})}, \quad \lambda_{2} = \frac {2(\beta_{2} - \beta_{0})} {\beta_{0} (c^{2}_{0} + c^{2}_{2})} .
\end{eqnarray*}
Hence, $\alpha_{2} = \lambda_{2}$ holds if and only if $c_{0} =
c_{1}$, which is indeed the skew symmetric condition. Thus,
  \eq{convolution indep} implies that the stationary distribution
  $\pi$ has exponential marginal distributions with parameters
  $\lambda_{1}$ and $\lambda_{2}$. However, it is unclear if there exist  parameters $(\beta_i,  c_i)$, $i=0,1,
    2$, for the tandem queue such that   \eq{convolution indep} is
    satisfied, but 
  $\lambda_{2} \ne \alpha_{2}$. Thus, it may be 
interesting to consider the following questions. 

\begin{question}
\label{que:class}
What is a class of SRBM satisfying \eq{convolution indep} ? How can we characterize this class in terms of the modeling primitives ? How much is it larger than the class satisfying the decomposability ?
\end{question}
\begin{question}
\label{que:approximation}
Can the stationary distributions of $Z^{K}(0)$ and $Z^{L}(0)$ serve good approximations for the marginal distributions of the original stationary distribution when \eq{convolution indep} does not hold ? If not, for what class of SRBM can they provide good approximations ?
\end{question}

Obviously, these two questions are closely related. \que{class} is
hard to answer while \que{approximation} may be studied through
numerical experiments. Furthermore, for $d=2$, we know the tail
asymptotics of the one-dimensional marginals and when their tail decay
rates are identical with $\theta^{(i,\ray)}_{i}$ defined in
(\ref{eq:thataray}) (see Theorem 2.2 and 
2.3 of \cite{DaiMiya2013}). This may suggest the class of SRBM for
which the product form approximation using $\theta^{(i,\ray)}_{i}$ is
reasonable. Unfortunately, we do not have
any explicit results yet for the tail decay rates for $d \ge 3$ except for
some special cases. We hope that this tail decay rate problem will be
solved  sometime in the future, and \que{approximation} will be better answered
then. 

Another question is about sufficient conditions for the
decomposability. We partially answered this question by
\thr{decomposition}. It seems hard to extend the arguments in the
proof of this theorem to more general cases. Such an extension is a 
challenging open problem.  

\section{Acknowledgments}

We are grateful to the anonymous referees for their
stimulating questions and helpful comments. This paper was initiated
when MM (the second author) visited JD (the first author) at Cornell
University in July, 2013. MM is grateful to the hospitality of School of
ORIE of Cornell University. JD is supported in part by NSF Grants
CMMI-1030589, CNS-1248117, and CMMI-1335724. MM is supported in part by
JSPS Grant 24310115.  

\bibliographystyle{elsarticle-harv}
\bibliography{dai09062014}{}

\appendix

\section{Some matrix classes}
\label{app:matrix}

We give definitions of some matrix classes in our arguments. One
  can find them in text books for matrices (e.g., see
  \cite{BermPlem1979,CottPangSton1992}). Let $A$ be an $n$-dimensional
  square matrix. Then, $A$ is called an $\sr{S}$-matrix if there is an
  $n$-dimensional vector $v > 0$ such that $A v > 0$, where, for a
  vector $u$, $u>0$
  means each component of $u$ is strictly positive.
 If all principal sub-matrices of $A$ are
  $\sr{S}$-matrices, then $A$ is called a completely-$\sr{S}$ matrix. If
  all principal minors of $A$ are positive, then $A$ is called a
  $\sr{P}$-matrix. Every $\sr{P}$-matrix is an $\sr{S}$-matrix (e.g,
  see Corollary 3.3.5 of \cite{CottPangSton1992}). A $\sr{P}$-matrix
is called an $\sr{M}$-matrix if 
  its off diagonal entries are all non-positive (see Definition 3.11.1
  of \cite{CottPangSton1992} for a $\sr{K}$-matrix, which is another
  name of an $\sr{M}$-matrix used in \cite[Section
  3.13.24]{CottPangSton1992}). 

\section{Proof of Corollary 1 and some related results}
\label{app:corollary 1}

To prove the corollary, we introduce some geometric objects:
\begin{eqnarray*}
 && E = \{\theta \in \dd{R}^{d}; \gamma(\theta) = 0\},\\
 && P^{(i)} =
  \cap_{k \in J \setminus \{i\}}\{\theta \in \dd{R}^{d};
  \gamma_{k}(\theta) = 0\}, \quad i \in J. 
\end{eqnarray*}
The object $E$ is an ellipse in $\dd{R}^d$.  Since $R$ is
invertible and $\theta \in P^{(i)}$ implies that $\br{\theta, R^{(k)}} =
0$ for $k \ne i$, $P^{(i)}$ must be a line going through the origin.
Clearly, for each $i$, $P^{(i)}$ intersects the ellipse $E$ by at most
two points, one of which is the origin. We denote its non-zero
intersection by $\theta^{(i,\ray)}$ if it exists. Otherwise, let
$\theta^{(i,\ray)} = 0$. The following lemma shows that the latter is
impossible by giving an explicit formula for $\theta^{(i,\ray)}$.
Recall that $(Q^{\rs{t}})^{(i)}$ be the $i$th column of
$Q^{\rs{t}}$. We refer to the following fact, obtained as Lemma 3 in
\cite{DaiMiyaWu2014a}. 

\begin{lemma} 
\label{lem:ray}
For each $i\in J$, 
\begin{equation}
  \label{eq:thataray}
  \theta^{(i,\ray)} = \Delta_i (Q^{\rs{t}})^{(i)},
\end{equation}
where $\Delta_i>0$ is defined in (\ref{eq:Deltai}).
\end{lemma}

\begin{proof}[Proof of \cor{marginal 1}]

It follows from
\cite{Harr1985} that the stationary distribution of a one-dimensional
SRBM with drift $\mu<0$ and variance $\sigma^2$ is exponential with mean
$1/\lambda$, where 
\begin{eqnarray}
\lambda = - \frac{2\mu} {\sigma^2}.
\label{eq:generalLambda}
\end{eqnarray}
We apply \thr{marginals} to $K = \{i\}$.  According to the theorem,
$Z^{\{i\}}(0)$ under $\pi$ is a one-dimensional SRBM with variance
$\tilde{\Sigma}(\{i\})$ and drift $\tilde{\mu}(\{i\})$. Set
\begin{eqnarray*}
  \lambda_{i} = - \frac {2\tilde{\mu}(\{i\})} {\tilde{\Sigma}(\{i\})}.
\end{eqnarray*}
By (\ref{eq:thataray}) and (\ref{eq:generalLambda}), to prove the
corollary it suffices to verify that 
\begin{equation}
  \label{eq:lambda_i}
  \lambda_i  = \theta^{(i,\ray)}_{i}. 
\end{equation}

 We first compute $Q$
and $Q^{(\{i\},\{i\})}$ for this. By \lem{ray}, we have 
\begin{eqnarray*}
  Q = (\Delta_{1}^{-1} \theta^{(1,\ray)}, \Delta_{2}^{-1} \theta^{(2,\ray)}, \ldots, \Delta_{d}^{-1} \theta^{(d,\ray)})^{\rs{t}},
\end{eqnarray*}
and therefore
\begin{eqnarray*}
  Q^{(\{i\},\{i\})} =\Delta_{i}^{-1} \theta^{(i,\ray)}_{i}.
\end{eqnarray*}
Hence,
\begin{eqnarray*}
 \lefteqn{\tilde{\Sigma}(\{i\}) = \Delta_{i}^{2} (\theta^{(i,\ray)}_{i})^{-2} \sum_{j,k} [Q]_{ij} \Sigma_{jk} [Q]_{ik}}\\
  && \quad = \Delta_{i}^{2} (\theta^{(i,\ray)}_{i})^{-2} \sum_{j,k} \Delta_{i}^{-1} \theta^{(i,\ray)}_{j} \Sigma_{jk} \Delta_{i}^{-1} \theta^{(i,\ray)}_{k} \\
  && \quad = (\theta^{(i,\ray)}_{i})^{-2} \br{ \theta^{(i,\ray)}, \Sigma \theta^{(i,\ray)}}\\
  && \quad = - 2 (\theta^{(i,\ray)}_{i})^{-2} \br{\theta^{(i,\ray)}, \mu},
\end{eqnarray*}
where the last equality is obtained since $\gamma(\theta^{(i,\ray)}) = 0$. Similarly, we have
\begin{eqnarray*}
  \tilde{\mu}(\{i\}) = (\theta^{(i,\ray)}_{i})^{-1} \br{\theta^{(i,\ray)}, \mu}.
\end{eqnarray*}
Hence,
\begin{eqnarray*}
\lambda_{i} = - \frac {2\tilde{\mu}(\{i\})} {\tilde{\Sigma}(\{i\})} = \theta^{(i,\ray)}_{i}.
\end{eqnarray*}
This completes the proof of \cor{marginal 1}.
\end{proof}

In what follows, we give a short proof for $\alpha_{i} = \lambda_{i}$ when the skew symmetric condition holds. For this, we use Lemma 2 in \cite{DaiMiyaWu2014a}, which characterizes the condition by the formula
\begin{eqnarray*}
  \gamma(\theta) = \sum_{i=1}^{d} \frac {\Sigma_{ii}} {2R_{ii}} \gamma_{i}(\theta) (\alpha_{i} - \theta_{i}), \qquad \theta \in \dd{R}^{d}.
\end{eqnarray*}
Then, we immediately have $\alpha_{i} = \theta^{(i,\ray)}_{i} = \lambda_{i}$ from the fact that $\gamma(\theta^{(i,\ray)}) = 0$, $\gamma_{i}(\theta^{(i, \ray)}) \ne 0$ and $\gamma_{i}(\theta^{(j,\ray)}) = 0$ for $j \ne i$.

\end{document}